\documentclass[10pt]{article}
\usepackage[nointlimits]{amsmath}
\usepackage{amsfonts}
\usepackage{amssymb}
\usepackage{amsmath}
\usepackage{amsthm}
\usepackage{latexsym}
\usepackage{graphicx}
\usepackage{layout}

\usepackage{amsfonts}
\usepackage{latexsym}
\usepackage{graphicx}
\usepackage{layout}
\usepackage{mathrsfs}
\usepackage{color}

\newtheorem{question}{Question}
\newtheorem{definition}{Definition}

\newtheorem{theorem}{Theorem}[section]
\newtheorem{proposition}[theorem]{Proposition}

\newtheorem{corollary}[theorem]{Corollary}
{\theoremstyle{definition}\newtheorem{remark}[theorem]{Remark}}

\newcommand{\nor}{\Arrowvert}

\def\R{{\rm I\mskip -3.5mu R}}
\def\N{{\rm I\mskip -3.5mu N}}

\def\Om{\Omega}
\def\C2al{C^{2,\a}_{loc}}

\def\intO{\int_{\Om}}

\def\na{\nabla}
\def\d{\delta}
\def\l{{\lambda}}
\def\a{{\alpha}}
\def\s{{\sigma}}
\def\S{{\mathcal{S}}}
\def\un{u_n}

\def\tv{\tilde v}
\def\tu{\tilde u}
\def\tx{\tilde x}
\def\ty{\tilde y}
\def\tz{\tilde z}
\def\c{\mathfrak{c}}

\def\C{\mathfrak{C}}
\def\ln{\l_n}
\def\L{\Lambda}
\def\vn{v_n}
\def\mn{\mu_n}

\def\de{\partial}

\usepackage{bm,comment}
\def\ka{\kappa}
\def\a{\mathfrak{a}}

\newcommand{\tend}{\longrightarrow}

\newcommand{\Co}[2]{C^{#1}\left(#2\right)}

\newcommand{\clo}[1]{\overline{#1}}

\newcommand{\ball}[2]{B_{#1}\left(#2\right)}
\newcommand{\til}[1]{\widetilde{#1}}

\begin{document}


\title{On the number of peaks of the eigenfunctions \\
 of the linearized Gel'fand problem
\thanks{The first two authors are supported by PRIN-2009-WRJ3W7 grant}}
\author{Francesca Gladiali  \thanks{Universit\`a degli Studi di
Sassari,via Piandanna 4 -07100 Sassari, e-mail {\sf
fgladiali@uniss.it}.} \and Massimo Grossi\thanks{Dipartimento di Matematica, Universit\`a di Roma
``La Sapienza", P.le A. Moro 2 - 00185 Roma, e-mail {\sf
grossi@mat.uniroma1.it}.}\and Hiroshi Ohtsuka \thanks{
Department of Applied Physics,
Faculty of Engineering,
University of Miyazaki,
Gakuen Kibanadai Nishi 1-1,
Miyazakishi, 889-2192, Japan,e-mail {\sf  ohtsuka@cc.miyazaki-u.ac.jp.}} }
\date{}
\maketitle
\begin{abstract}
We derive a second order estimate 
for the first $m$ eigenvalues and eigenfunctions of the linearized Gel'fand problem associated to solutions which blow-up at $m$ points.
This allows us to determine, in some suitable situations, some qualitative properties of the first $m$ eigenfunctions as
the number of points of concentration or the multiplicity of the eigenvalue .
\end{abstract}


\section{Introduction and statement of the main results} 
\label{s1}
Let us consider the Gel'fand problem,
\begin{equation}
\label{1}
            \left\{\begin{array}{lc}
                        -\Delta u=\l e^{u}  &
            \mbox{  in }\Om\\
              u=0 & \mbox{ on }\partial \Om,
                      \end{array}
                \right.
\end{equation}
where $\Om\subset \R^2$ is a bounded domain with smooth boundary $\de \Om$ and $\l>0$ is a real parameter. This problem appears in a wide variety of areas of mathematics such as the conformal embedding of a flat domain into a sphere \cite{B80}, self-dual gauge field theories \cite{JT80}, equilibrium states of large number of vortices \cite{jm73,pl76,clmp92,kie93,clmp95}, stationary states of chemotaxis motion \cite{SS00}, and so forth. See \cite{GGOS12} for more about our motivation and \cite{suzuki08} for other background materials.

Let $\{\ln\}_{n\in\N}$ be a sequence of positive values such that $\ln\to 0$ as $n\to \infty$ and let $\un=\un(x)$ be a sequence of solutions of \eqref{1} for $\l=\ln$. In \cite{NS90}, the authors studied solutions $\{ u_{n}\}$ which blow-up at $m$-points (see next section for more details).
This means that there is a set ${\cal S}=\{ \kappa_1, \cdots,
\kappa_m\}\subset \Omega$ of $m$ distinct points such
that\\
$i)\ \nor \un \nor_{L^{\infty}(\omega)}=O(1)
 \quad\hbox{for any }
\omega\Subset \overline{ \Omega}\setminus \S,$\\
$ii)\ \un{|_{\S}}\rightarrow +\infty \quad \hbox{ as }n\to \infty.$\\
In  \cite{NS90b}, \cite{EGP05} and \cite{DKM05}
some sufficient conditions which ensure the existence of this type of solutions are given.

Throughout the paper we will consider solutions $u_n$ to \eqref{1} with $m$ blow-up  points and we investigate the eigenvalue problem
\begin{equation}\label{autov-n}
\left\{\begin{array}{lc}
                        -\Delta \vn^k= \mn^k\ln e^{\un}\vn^k  & \mbox{ in }\Om\\
\nor \vn^k\nor_{\infty}=\max_{\overline{ \Om}}\vn^k=1 &\\
                  \vn^k=0 & \mbox{ on }\partial \Om
                      \end{array}
                \right.
\end{equation}
which admits a sequence of eigenvalues $\mn^1<\mn^2\leq\mn^3\leq  \dots$, where $\vn^k$ is the $k$-th eigenfunction of \eqref{autov-n} corresponding to the eigenvalue $\mn^k$.
 In order to state our results we need to introduce some notations and recall some well known facts.
 
 Let $R>0$  be such that $B_{2R}(\kappa_i)\subset\subset \Omega$ for $i=1,\dots,m$ and $B_R(\kappa_i)\cap B_R(\kappa_j)=\emptyset$ if $i\neq j$.
For each $\kappa_j\in \S$ there exists a sequence $\{x_{j,n}\}\in B_R(\kappa_j)$ such that
\[
\un(x_{j,n})=\sup_{B_R(x_{j,n})}\un(x)\rightarrow +\infty\quad\text{and}\quad
x_{j,n}\to \kappa_j\quad\text{as $n\to +\infty$}.
\]
For any $j=1,\dots,m$, we rescale $\un$ around $x_{j,n}$, letting
\begin{equation}\label{2.4}
\tu_{j,n}(\tx):= \un\left( \d_{j,n} \tx +x_{j,n}\right)- \un(x_{j,n})\quad \hbox{ in }B_{\frac{R}{\d_{j,n}}}(0),
\end{equation}
where the scaling parameter $\d_{j,n}$ is determined by
\begin{equation}\label{2.5}
\ln e^{\un(x_{j,n})}\d_{j,n}^2=1.
\end{equation}
It is known that $\delta_{j,n}\tend 0$ and for any $j=1,\dots,m$
\begin{equation}\label{2.6}
\tu_{j,n}(\tx)\rightarrow U(\tx)=\log \frac 1{\left( 1+\frac{|\tx|^2}8\right)^2} \quad \hbox{ in }C^{2,\alpha}_{loc}(\R^2).
\end{equation}

As we did for $\un$ we rescale also the eigenfunctions $\vn^k$ around $x_{j,n}$ for any $j=1,\dots,m$. So we define
\begin{equation}\label{2.8}
\tv_{j,n}^k(\tx):=\vn^k\left( \d_{j,n} \tx +x_{j,n}\right)\quad \hbox{ in }B_{\frac{R}{\d_{j,n}}}(0),
\end{equation}
where $\d_{j,n}$ is as in \eqref{2.5}. 
The rescaled eigenfunctions $\tv_{j,n}^k(\tx)$ satisfy
\begin{equation}\label{2.8a}
\left\{
\begin{array}{ll}
-\Delta \tv_{j,n}^k=\mn ^k e^{\tu_{j,n}}\tv_{j,n}^k & \hbox{ in }B_{\frac{R}{\d_{j,n}}}(0)\\
\nor \tv_{j,n}^k \nor_{L^{\infty}\big(B_{\frac{R}{\d_{j,n}}}(0)\big)}\leq 1.
\end{array}
\right.
\end{equation}
One of the main results of this paper concerns pointwise estimates of the eigenfunction. In particular, we are interested in the number of peaks of
$\vn^k$ for $k=1,..,m$. Let us recall that, by Corollary 2.9 in \cite{GGOS12}, we have that
$$\vn^k\rightarrow0\quad \text{in $\Co{1}{\clo{\Omega}\backslash\cup_{j=1}^m\ball{R}{\kappa_j}}$}$$
This means that $\vn^k$ can concentrate only at $\kappa_j$, $j=1,..,m$.
This leads to the following definition,
\begin{definition}\label{de1} 
We say that an eigenfunction $\vn^k$ concentrates at $\kappa_j\in\Omega$ if there exists $\kappa_{j,n}\rightarrow \kappa_j$ such that
\begin{equation}\label{i1}
\left|\vn^k(\kappa_{j,n})\right|\ge C>0\quad\hbox{for }n\hbox{ large}.
\end{equation}
\end{definition}

A problem that arises naturally is the following,
\begin{question}\label{q1}
Let us suppose that $u_n$ blows-up at the points $\left\{k_1,..,k_m\right
\}$. 
Is the same true for the eigenfunction $\vn^k$, $k=,1..,m$ associated to a simple eigenvalue $\mn ^k$ of \eqref{autov-n}?
\end{question}
Obviously if the eigenvalue $\mn ^k$ is multiple, in general it makes no sense to speak about the number of point of concentration, since this depends
on the linear combination of the eigenfunctions.

A first partial answer related to this question was given in \cite{GGOS12}, where the following result was proved.
\begin{theorem}
\label{t3}
For each $k\in \{1,\dots,m\}$ there exists a vector
\begin{equation}
\c^k=(c_1^k,\dots,c_m^k)\in [-1,1]^m\subset \R^m,\quad \c^k\not=\bm{0}
\label{2.9a}
\end{equation}
such that for each $j\in \{1,\dots,m\}$, there exists a subsequence satisfying
\begin{equation}
\label{2.9}
\tv_{j,n}^k(x)\to c_j^k \quad \text{ in }C^{2,\alpha}_{loc}\left(\R^2\right)
\end{equation}
\begin{equation}
\label{2.11a}
\c^k\cdot\c^h=o\quad\hbox{if }h\ne k
\end{equation}
and
\begin{equation}
\label{2.11}
\frac {\vn^k}{\mn^k}\to 8\pi \sum_{j=1}^m c_j^k \, G(\cdot,\kappa_j)\quad \text{ in }C^{2,\alpha}_{loc}\left( \overline{\Om} \setminus \{\kappa_1,\dots,\kappa_m\}\right).
\end{equation}
\end{theorem}
Here $G(x,y)$ denotes the Green function of $-\Delta$ in $\Om$ with Dirichlet boundary condition,
i.e.
\begin{equation}
\label{2.3}
G(x,y)=\frac 1{2\pi}\log{|x-y|^{-1}}+K(x,y),
\end{equation}
 $K(x,y)$ is the regular part of $G(x,y)$ and
 $R(x)=K(x,x)$ the Robin
function.
A consequence of Theorem \ref{t3} is that
\begin{equation}
\label{j8}
\vn^k\hbox{ concentrates at }\kappa_j\hbox{ if and only if }c_j^k\ne0.
\end{equation}
In this paper we characterize the values $c_j^k$ in term of the Green function and this will allow us to determine whether $c_j^k$ is equal to $0$ or not.
\begin{theorem}\label{t2}
For each $k\in \{1,\dots,m\}$, we have that\\
$i)$ The vector
$\c^k=(c_1^k,\dots,c_m^k)\in [-1,1]^m\subset \R^m\backslash\{\bm{0}\}$ is the $k$-th eigenvector
of the matrix
\begin{equation}\label{matrix_h}
h_{ij}=\left\{
\begin{array}{ll}
R(\kappa_i)+2\sum_{\substack   {1\leq h\leq m\\ h\neq i}} G(\kappa_h,\kappa_i) & \text{if }i=j,\\
-G(\kappa_i,\kappa_j)&\text{if }i\neq j,
\end{array}\right.
\end{equation}
$ii)$ A sub-sequence of $\{\vn^k\}$ satisfies
\begin{align}
&\label{2.12} \tv_{j,n}^k(\tx)=\vn ^k(x_{j,n})+\mn^k c_j^k U(\tx)+o\left(\mn^k\right)\quad \text{ in }C^{2,\alpha}_{loc}\left(\R^2\right)
\end{align}
for each $j\in \{1,\dots,m\}$, where $U(\tx)$ is as defined in \eqref{2.6}.
\end{theorem}
Let us observe that \eqref{2.12} is a second order estimates for $\vn^k$. We stress that this is new even for the case of one-peak solutions ($k=1$).
From Theorem \ref{t2} we can deduce the answer to the Question \ref{q1},
\begin{corollary}\label{j12}
Let $\c^k=(c^k_1,..,c^k_m)$ be the $k$-th eigenvector of the matrix $(h_{ij})$. Then if $\mu_n^k$ is simple and  if $c_j^k\ne0$ we have that
 $v_n^k$ concentrates at $k_j$.
\end{corollary}
Our next aim is to understand better when $c_j^k\ne0$.
The following proposition gives some information in this direction.
\begin{theorem}\label{j2}
Let $k\in \{1,\dots,m\}$, $\mn ^k$ a simple eigenvalue and $v_n^k$ the corresponding eigenfunction.
Then we have that,\\
$i)$ any  $v_n^1$ concentrates at $m$ points $\kappa_1,..,\ \kappa_m$,\\
$ii)$ any $v_n^k$ concentrates at least at $two$ points $\kappa_i$, $\kappa_j$ with $i,j\in\{1,..,m\}$.
\end{theorem}
However, there are other interesting
questions. One is the following:
\begin{question}\label{q2}
Let us suppose that $\mn ^k$  is a multiple eigenvalue of \eqref{autov-n}. What about its multiplicity?
\end{question}
We will give an answer to this question in the case where $\Omega$ is an annulus.
\begin{theorem}\label{j3}
Let $\Omega$ be an annulus, $V^k_n$ the eigenspace associated to
$\mn^k$ and $dim\left(V^k_n\right)$ denote its dimension. Then,
\begin{itemize}
\item if $m$ is odd then $dim\left(V^k_n\right)\ge2$ for any $k\ge2$.
\item If $m$ is even then there is exactly one simple eigenvalue $\mn^{\bar k}$  for ${\bar k}\ge 2$
with eigenvector
$\c^{\bar k}=(-1,1,-1,1,..,-1.1)$ and all the other eigenvalues satisfy $dim\left(V^k_n\right)\ge2$ for any $k\ge2,\ k\ne\bar k$.
\end{itemize}
\end{theorem}

The previous results are a consequence of the following theorem,
which is a refinement up
the second order of some estimates proved
of \cite{GGOS12}. In our opinion this result is interesting in
itself.
\begin{theorem}\label{t1}
For each $k\in \{1,\dots,m\}$, it holds that
\begin{equation}\label{2.7}
\mn^k=-\frac 12 \frac 1{\log \ln}+\left( 2\pi \L^k-\frac {3\log 2-1}{2}\right)\frac 1{\left( \log \ln\right)^2}+o\left(\frac 1{\left( \log \ln\right)^2}\right)
\end{equation}
as $n\to +\infty$, where $\L^k$ is the $k$-th eigenvalue of the $m\times m$ matrix $(h_{ij})$ defined in \eqref{matrix_h} assuming $\L^1\leq \dots\leq \L^m$.
\end{theorem}
So the effect of the domain $\Omega$ on the eigenvalues $\mn^k$
appears in the second order term of the expansion of $\mn^k$.

The paper is organized as follows: in Section \ref{s0} we give some definitions and we recall some known facts. In Section \ref{s4} we prove Theorem \ref{t1} and some results on the vector  $\c^k$ introduced in Theorem \ref{t2}. In Section \ref{s5} we complete the proof of Theorem \ref{t2} and prove Theorem \ref{j2} and Theorem \ref{j3}.
\section{Preliminaries and known facts}
\label{s0}
Let us recall some results about the asymptotic behavior of $\un=\un(x)$ as $n\to +\infty$.
In \cite{NS90}, the authors proved that, along a sub-sequence,
\begin{equation}\label{2.1a}
\ln \intO e^{\un}\, dx\rightarrow 8\pi m
\end{equation}
for some $m=0,1,2, \cdots, +\infty$. Moreover
\begin{itemize}
\item
If $m=0$ the pair $({\lambda_n},u_{\lambda_n} )
$
converges to
$(0,0)$ as ${\lambda_n}\rightarrow0$.
\item
If $m=+\infty$ it holds the entire blow-up of the solution
$u_n$, i. e.
$\inf_Ku_n\rightarrow+\infty$ for any $K\Subset \Omega$.
\item
If $0<m<\infty$ the solutions $\{ u_{n}\}$ blow-up at $m$-points.
Thus there is a set ${\cal S}=\{ \kappa_1, \cdots,
\kappa_m\}\subset \Omega$ of $m$ distinct points such
that  $\nor \un \nor_{L^{\infty}(\omega)}=O(1)$ for any
$\omega\Subset \overline{ \Omega}\setminus \S$,
\[
\un{|_{\S}}\rightarrow +\infty \quad \hbox{ as }n\to \infty,
\]
and
\begin{equation}\label{2.2}
\un \rightarrow \sum_{j=1}^m 8\pi \,G(\cdot, \ka_j)\quad \hbox{ in }C^{2}_{loc}(\overline{ \Omega} \setminus \S).
\end{equation}
\end{itemize}
In \cite{NS90}, it is also proved that the blow-up points ${\cal S}=\{ \kappa_1, \cdots, \kappa_m\}$ satisfy
\begin{equation}
\na H^m (\ka_1,\dots ,\ka_m)=0
\label{conditionS}
\end{equation}
where
\[
H^m(x_1,\dots, x_m)=\frac 12 \sum_{j=1}^m R(x_j)+\frac 12 \sum_{\substack   {1\leq j,h\leq m\\ j\neq h}}G(x_j,x_h).
\]
Here $H^m$ is  the \emph{Hamiltonian} function of the theory of vortices with equal intensities, see \cite{jm73,pl76,clmp92,kie93,clmp95} and references therein.

As we did in the introduction, let $R>0$  be such that $B_{2R}(\kappa_i)\subset\subset \Omega$ for $i=1,\dots,m$ and $B_R(\kappa_i)\cap B_R(\kappa_j)=\emptyset$ if $i\neq j$ and $x_{j,n}$, $\un$, $\tu_{j,n}$ and $\d_{j,n}$ as in \eqref{2.4}, \eqref{2.5}. In \cite{GOS11}, Corollary 4.3, it is shown that there exists a constant $d_j>0$ such that
\begin{equation}\label{2.6b}
\d_{j,n}=d_j \ln^{\frac 12}+o\left( \ln^{\frac 12}\right)
\end{equation}
as $n\to \infty$ for a sub-sequence, and in particular, $\delta_{j,n}\tend 0$.
In \cite{GOS11} the exact value of $d_j$ was not computed, but for our aim it is crucial to have it. We will give it in
\eqref{value_dj}. From \eqref{2.5} and \eqref{2.6b} we have
\begin{equation}\label{2.6c}
\un(x_{j,n})=-2 \log \ln-2\log d_j +o(1)
\end{equation}
as $n\to \infty$ for any $j=1,\dots,m$.

The function $\tu_{j,n}$ defined in the Introduction satisfies
\begin{equation}\nonumber
\left\{
\begin{array}{ll}
-\Delta \tu_{j,n}=e^{\tu_{j,n}} & \hbox{ in }B_{\frac{R}{\d_{j,n}}}(0)\\
\tu_{j,n}\leq \tu_{j,n}(0)=0& \hbox{ in }B_{\frac{R}{\d_{j,n}}}(0).
\end{array}
\right.
\end{equation}
Using the result of \cite{CL91}, it is easy to see that, for any $j=1,\dots,m$
\begin{equation}\label{2.6z}
\tu_{j,n}(\tx)\rightarrow U(\tx)=\log \frac 1{\left( 1+\frac{|\tx|^2}8\right)^2} \quad \hbox{ in }C^{2,\alpha}_{loc}(\R^2).
\end{equation}
Moreover, it holds
\begin{equation}\label{2.6a}
\big| \tu_{j,n}(\tx)- U(\tx)\big|\leq C\quad \forall \tx \in B_{\frac{R}{\d_{j,n}}}(0)
\end{equation}
for any $j=1,\dots,m$ for a suitable positive constant $C$, see \cite{YYL99}.

Let us consider the eigenfunction $v_n^k$ defined in \eqref{autov-n} and recall the following result:
\begin{theorem}[\cite{GGOS12}]\label{t1a}
For $\lambda_n\rightarrow0$, it holds that
\begin{equation}\label{10}
 \mn^k=-\frac 12 \frac 1{\log \ln}+o\left(\frac 1{\log \ln}\right)\quad\text{for $1\leq k\leq
 m$},
\end{equation}
\begin{equation}\label{10_1}
 \mn^{k}=1-48\pi\eta^{2m-(k-m)+1}\ln+o\left(\ln\right)\quad\text{for $m+1\leq k\leq 3m$},
\end{equation}
and
\begin{equation}\label{10_2}
\mn^k>1\quad\text{for $k\geq 3m+1$}
\end{equation}
where $\eta^k$\,$(k=1,\cdots,2m)$ is the $k$-th eigenvalue of the
matrix $D(\mathrm{Hess}H^m)D$ at $(\ka_1,\cdots,\ka_m)$. Here $D=(D_{ij})$ is the diagonal matrix $\mathrm{diag}[d_1,d_1,d_2,d_2,\cdots,d_m,d_m]$
 (see \eqref{2.6b} for
the definition of the constants $d_j$ and \eqref{value_dj} for the precise value of it).
\end{theorem}
One of the purpose of this paper is to refine \eqref{10} (see Theorem \ref{t1} in the introduction).
\section{Fine behavior of eigenvalues}
\label{s4}
We start from the following proposition, which plays a crucial role in our argument.
\begin{proposition}\label{p4.2}
For any $k=1,\dots,m$ we have
\begin{align}
&\left\{\frac 1{\mn^k}-\un(x_{j,n})\right\}\ln \int_{B_R(x_{j,n})}\!\!\!\!\!\!\!\!\!e^{\un}\vn^k \, dx
\nonumber\\
&\qquad=(8\pi)^2 \sum_{\substack   {1\leq i\leq m\\ i\neq j}}(c_i^k-c_j^k)G(\kappa_j,\kappa_i)-16\pi c_j^k +o(1).
\label{4.18}
\end{align}
\end{proposition}
\begin{proof}
From \eqref{1} and \eqref{autov-n}, we have
\begin{align}
\int_{\de B_R(x_{j,n})} \Big\{ \Big.& \frac{\de \un}{\de \nu}  \frac{\vn^k}{\mn^k}-\un \frac{\de }{\de \nu}\left( \frac {\vn^k}{\mn^k}\right)\Big.\Big\}\, d\sigma\nonumber\\
=&\int_{ B_R(x_{j,n})} \left\{ \Delta \un  \frac{\vn^k}{\mn^k}-\un \Delta \frac {\vn^k}{\mn^k}\right\}\, dx\nonumber\\
=&-\frac 1{\mn^k}\ln \int_{ B_R(x_{j,n})}e^{\un}\vn^k \, dx+\ln \int_{ B_R(x_{j,n})}e^{\un}\vn^k \un\, dx\nonumber\\
=&-\frac 1{\mn^k}\ln \int_{ B_R(x_{j,n})}e^{\un}\vn^k \, dx+\un(x_{j,n})\ln \int_{ B_R(x_{j,n})}e^{\un}\vn^k \, dx\nonumber\\
&+\int_{B_{\frac R{\d_{j,n}}}(0)} e^{\tu_{j,n}}\tv^k_{j,n}\tu_{j,n}\, d\til{x}\label{4.1a}
\end{align}
and
\begin{equation}\label{4.2}
\int_{B_{\frac R{\d_{j,n}}}(0)} e^{\tu_{j,n}}\tv^k_{j,n}\tu_{j,n}\, dx\to\int_{\R^2}e^Uc_j^k U\, dx= -16 \pi c_j^k.
\end{equation}
On the other hand, from \eqref{2.2} and \eqref{2.11}, we have
 \begin{align}
&\int_{\de B_R(x_{j,n})}  \Big\{\frac{\de \un}{\de \nu} \frac{\vn^k}{\mn^k}-\un \frac{\de }{\de \nu}\left( \frac {\vn^k}{\mn^k}\right)\Big\}\, d\sigma\nonumber\\
&\to(8\pi)^2 \sum_{i=1}^m \sum_{h=1}^m c_h^k\int_{\de B_R(\kappa_{j})} \!\!\!\left\{\frac{\de}{\de \nu}G(x,\kappa_i)G(x,\kappa_h)-G(x,k_i)\frac{\de }{\de \nu} G(x,\kappa_h)\right\}\, d\sigma.\label{4.3}
\end{align}
We let
\[
I_{i,h}=\int_{\de B_R(\kappa_{j})} \left\{\frac{\de}{\de \nu}G(x,\kappa_i)G(x,\kappa_h)-G(x,\kappa_i)\frac{\de }{\de \nu} G(x,\kappa_h)\right\}\, d\sigma.
\]
Then we have\\
\noindent \underline{case 1}: $i=h$\\
\[
I_{i,h}=0.
\]
\noindent \underline{case 2}: $i\neq h$\\
In this case we have
\begin{align*}
I_{i,h}&=\int_{B_R(\kappa_{j})} \left\{\Delta G(x,\kappa_i) G(x,\kappa_h)- G(x,\kappa_i) \Delta G(x,\kappa_h)\right\}\, d\sigma\\
&=-G(\kappa_j,\kappa_h)\d_i ^j+G(\kappa_j,\kappa_i)\d_j^h
\end{align*}
where  $\d_a^b=1$ if $a=b$ and $\d_a^b=0$  else.

Therefore, from \eqref{4.3} we have
\begin{align}
\int_{\de B_R(x_{j,n})}& \left\{ \frac{\de \un}{\de \nu} \frac{\vn^k}{\mn^k}-\un \frac{\de }{\de \nu}\left( \frac {\vn^k}{\mn^k}\right)\right\}\, d\sigma\nonumber\\
&=(8\pi)^2\sum_{i=1}^m \sum_{\substack   {1\leq h\leq m\\ h\neq i}} c_h^k\left\{-G(\kappa_j,\kappa_h)\d_i^j+G(\kappa_j,\kappa_i)\d_j^h\right\}+o(1)\nonumber\\
&=(8\pi)^2\Big\{-\sum_{\substack   {1\leq h\leq m\\ h\neq j}} c_h^kG(\kappa_j,\kappa_h)+\sum_{\substack   {1\leq i\leq m\\ i\neq j}} c_j^kG(\kappa_j,\kappa_i)\Big\}+o(1)\nonumber\\
&=-(8\pi)^2 \sum_{\substack   {1\leq i\leq m\\ i\neq j}} \left(c_i^k-c_j^k\right) G(\kappa_j,\kappa_i)+o(1).\label{4.4}
\end{align}
The proof follows from \eqref{4.1a}, \eqref{4.2}, and \eqref{4.4}.
\end{proof}
Next we are going to get the precise value of $d_j$ in \eqref{2.6c}. To this purpose we need to strengthen \eqref{2.6c}.
\begin{proposition}[(cf. Estimate D in \cite{CL02})]
\label{p4.1}
Let $\un$ be a solution of \eqref{1} corresponding to $\ln$, and let $x_{j,n}$ and $R$ be as in Section \ref{s1}. Then, for any $j=1,\dots,m$ we have
\begin{equation}\label{norm-infty}
\un(x_{j,n})=-\frac{\s_{j,n}}{\s_{j,n}-4\pi} \log \ln -8\pi \Big\{ R(x_{j,n})+\sum_{\substack   {1\leq i\leq m\\ i\neq j}} G(x_{j,n},x_{i,n})\Big\}+6\log 2 +o(1)
\end{equation}
where
\begin{equation}\label{sigma}
\s_{j,n}=\ln \int_{B_R(x_{j,n})}e^{\un}\, dx \to 8\pi.
\end{equation}
\end{proposition}
\begin{proof}
Using the Green representation formula, from \eqref{1}, we have
\begin{align*}
\un(x_{j,n})=& \intO G(x_{j,n},y) \ln e^{\un(y)}\, dy\\
=&\frac 1{2\pi} \int_{B_R(x_{j,n})}\log |x_{j,n}-y|^{-1}\ln e^{\un(y)}\, dy\\
&+\int_{B_R(x_{j,n})}K(x_{j,n},y)\ln e^{\un(y)}\, dy\\
&+\sum_{\substack{1\leq i\leq m \\ i\neq j}}\int_{B_R(x_{i,n})}G(x_{j,n},y) \ln  e^{\un(y)}\, dy\\
&+\int_{\Omega \setminus \bigcup_{i=1}^mB_R(x_{i,n})} G(x_{j,n},y) \ln  e^{\un(y)}\, dy\\
=& -\frac{\s_{j,n}}{2\pi}\log \d_{j,n}+\frac 1{2\pi}\int_{B_{\frac R{\d_{j,n}}}(0)} \log |\tilde{y}|^{-1}e^{\tu_{j,n}(\tilde{y})}\, d\tilde{y}\\
&+ 8\pi \Big\{R(x_{j,n})+\sum_{\substack{1\leq i\leq m\\ i\ne\ j}}G(x_{j,n},x_{i,n})\Big\}+o(1).
\end{align*}
Using the estimate \eqref{2.6a}, we get here
\begin{equation}
\frac 1{2\pi}\int_{B_{\frac R{\d_{j,n}}}(0)}\log
|\tilde{y}|^{-1}e^{\tu_{j,n}(\tilde{y})}\, d\tilde{y}\to \frac
1{2\pi}\int_{\R^2}\log |\tilde{y}|^{-1}e^{U(\tilde{y})}\,
d\tilde{y}=-6\log 2. \label{A_0}
\end{equation}
Then the conclusion follows by  \eqref{2.5} and \eqref{sigma}.
\end{proof}
Here we recall a fine behavior of the local mass $\s_{j,n}$ defined in \eqref{sigma}.
\begin{proposition}\label{p5.1}
For any $j\in \{1,\dots,m\}$ we have
\begin{equation}
\label{middle-est}
\s_{j,n}=8\pi +o\left(\ln^\frac{1}{2}\right)
\end{equation}
\end{proposition}
\begin{proof}
see \cite[Remark 5.6]{O12}
\end{proof}
\begin{remark}
 We note that a stronger version
\begin{equation}
\label{sharp-est}
\s_{j,n}=8\pi +o\left(\ln\right)
\end{equation}
follows from (3.56) of \cite{CL02}. However, for our aims, it is sufficient to use the estimate
\eqref{middle-est}.
\end{remark}
Using Proposition \ref{p4.1} and Proposition \ref{p5.1}, we get the precise value of $d_j$ given in \eqref{2.6b}.
\begin{proposition}\label{newp5.5}
For any $j=1,..,k$ it holds,
\begin{equation}
d_j=\frac{1}{8}\exp\left\{ 4\pi R(\kappa_j)+4\pi\sum_{\substack   {1\leq i\leq m\\ i\neq j}}G(\ka_j,\ka_i)\right\}.
\label{value_dj}
\end{equation}
\end{proposition}
\begin{proof}
From \eqref{norm-infty}, we get
\begin{align}
\un(x_{j,n})=&-2\log\ln+\frac{\s_{j,n}-8\pi}{\s_{j,n}-4\pi}\log \ln
\nonumber\\
&\quad-8\pi \Big\{R(\kappa_{j})+\sum_{\substack{1\leq i\leq m\\ i\ne\ j}}G(\kappa_{j},\kappa_{i})\Big\}+6\log 2+o(1).
\label{fine_u}
\end{align}
From \eqref{middle-est} it follows that $\frac{\s_{j,n}-8\pi}{\s_{j,n}-4\pi}\log \ln=o(1)$. Therefore the claim follows from \eqref{2.6c}.
\end{proof}
As a consequence of \eqref{2.6c} and Proposition \ref{newp5.5}, we get, using \eqref{4.18}
\begin{align}
&\left\{\frac 1{\mn^k}+
2\log \ln\right\}\int_{B_R(x_{j,n})}\!\!\!\!\!\!\!\ln  e^{\un}\vn^k \, dx=(8\pi)^2 \sum_{\substack{1\leq i\leq m \\i\neq j}}c_i^kG(\kappa_j,\kappa_i)\nonumber\\
&\quad-(8\pi)^2c_j^k \Big\{R(\kappa_j)+2\sum_{\substack{1\leq i\leq m \\i\neq j}}G(\kappa_j,\kappa_i)\Big\}+48 \pi c_j^k \log 2  -16 \pi c_j^k +o(1)
\nonumber\\
&=-(8\pi)^2\sum_{i=1}^m h_{ji}c_i^k+16\pi c_j^k(3\log 2-1)+o(1),\label{4.26}
\end{align}
(see the definition of the matrix $(h_{ij})$ in Theorem \ref{t1}).
\begin{proposition}\label{p5.2}
For any $j,h\in \{1,\dots,m\}$ it holds that
\begin{align}\label{5.1}
c_h^k\sum_{i=1}^m h_{ji}c_i^k=c_j^k\sum_{i=1}^m h_{hi}c_i^k
\end{align}
\end{proposition}
\begin{proof}
Multiplying $\int_{B_R(x_{h,n})} \ln e^{\un}\vn^k \, dx $ to \eqref{4.26} and $\int_{B_R(x_{j,n})} \ln e^{\un}\vn^k \, dx $ to \eqref{4.26} with $j=h$, and then subtracting the latter from the former, we get the conclusion from \eqref{2.6} and \eqref{2.9}.
\end{proof}
\begin{proposition}\label{p5.3}
The vector $\c^k$, defined in \eqref{2.9a}, is an eigenvector of $(h_{ij})$.
\end{proposition}
\begin{proof}
First we assume that there are $c_j^k\neq 0$ and $c_h^k\neq 0$ for $j\neq h$. Then \eqref{5.1} gives
\begin{equation}\label{a}
\frac 1{c_j^k} \sum_{i=1}^m h_{ji}c_i^k=\frac 1{c_h^k} \sum_{i=1}^m h_{hi}c_i^k=\L^k.
\end{equation}
for all $j$ satisfying $c_j^k\neq 0$. Then $\L^k$ is an eigenvalue of $(h_{ij})$ if $c_j^k\neq 0$ for all $j=1,\dots,m$. 

On the other hand, for $j\in \{1,\dots,m\}$ satisfying $c_j^k=0$, we can choose $c_h^k\neq 0$ (see \eqref{2.9}) so that
\begin{equation}\label{5.30}
\sum_{i=1}^m h_{ji}c_i^k=0\quad \hbox{ if }c_j^k=0.
\end{equation}
From \eqref{a} and \eqref{5.30}, we get that $\c^k$ is an eigenvector of $(h_{ij})$ if there are at least two $j$ satisfying $c_j^k\neq 0$.

The last case is that there is only one $j$ satisfying $c_j^k\neq 0$, but this never happens. Indeed, in this case \eqref{5.1} becomes
\[
\sum_{i=1}^mh_{hi}c_i^k=h_{hj}c_j^k=0 \quad (j\neq h)
\]
which contradicts $h_{hj}=-G(\kappa_h,\kappa_j)\neq 0$.
\end{proof}
\begin{proof}[Proof of Theorem \ref{t1}]
Take $c_j^k\neq 0$. Then Proposition \ref{p5.3} implies that $\sum_{i=1}^m h_{ji}c_i^k=\L^k c_j^k$ and therefore \eqref{4.26} implies that
\begin{equation}\label{L}
\frac 1{\mn^k}=-2\log \ln-8\pi\Lambda^k+2(3\log 2-1)+o(1).
\end{equation}
Indeed, letting $L=-8\pi\Lambda^k+2(3\log 2-1)$, \eqref{L} leads that
\begin{align}
\mn^k&=\frac{1}{-2\log\ln+L+o(1)}=-\frac{1}{2\log\ln}\cdot\frac{1}{1-\frac{L+o(1)}{2\log\ln}}
\nonumber\\
&=-\frac{1}{2\log\ln}\left\{1+\frac{L+o(1)}{2\log\ln}+o\left(\frac{1}{\log\ln}\right)\right\}
\nonumber\\
&=-\frac{1}{2\log\ln}-\frac{L}{4}\cdot\frac{1}{(\log\ln)^2}+o\left(\frac{1}{(\log\ln)^2}\right).
\label{-L/4}
\end{align}
Therefore \eqref{2.7} follows.

The formula \eqref{2.7} gives $\L^1\leq \dots\leq \L^m$, since $\mn^1<\mn^2\leq \dots\leq \mn^m$. Consequently we get $\L^k$ is the $k$-th eigenvalue. Since $\L^k$ depends only on $(h_{ij})$ then equation \eqref{2.7} holds without taking a sub-sequence.
\end{proof}
\section{Fine behavior of eigenfunctions}
\label{s5}
We start this section with the following
\begin{proposition}\label{p4.3}
For any $k,j\in\{1,\dots,m\}$, we have
\begin{align}
\frac{\vn^k(x_{j,n})}{\mn^k}&=\frac 1{2\pi}\log \d_{j,n}^{-1} \int_{B_R(x_{j,n})}\ln e^{\un(y)} \vn^k(y)\, dy\nonumber\\
&+8\pi \Big\{c_j^kR(\kappa_j)+\sum_{\substack{1\leq i\leq m\\ i\neq j}}c_i^kG(\kappa_j,\kappa_i)\Big\} -6c_j^k\log 2  +o(1).\label{4.23}
\end{align}
\end{proposition}
\begin{proof}
Using the Green representation formula and \eqref{autov-n}, we have, as in the proof of the Proposition \ref{p4.1}
\begin{align*}
\frac{\vn^k(x_{j,n})}{\mn^k}=&\intO G(x_{j,n},y) \ln e^{\un(y)}\vn^k(y) \, dy\\
=& \frac 1{2\pi} \log \d_{j,n}^{-1}  \int_{B_R(x_{j,n})}\ln  e^{\un(y)}\vn^k(y) \, dy\\
&+\frac 1{2\pi } \int_{B_{\frac R{\d_{j,n}}}(0)} \log |\ty|^{-1} e^{\tu_{j,n}(\ty)}\tv^k_{j,n}(\ty) \, d\ty\\
&+ \Big\{ 8\pi c_j^k R(\kappa_j)+8\pi \sum_{\substack{1\leq i\leq m\\ i\neq j}}c_i^kG(\kappa_j,\kappa_i)\Big\}+o(1)
\end{align*}
and the claim follows.
\end{proof}
\begin{remark}
From \eqref{norm-infty}, \eqref{2.5},
and Proposition \ref{p5.1}, we get
\begin{align}
&\frac{1}{\mn^k}\int_{B_R(x_{j,n})}\ln  e^{\un}\vn^k(x_{j,n})\, dx=\frac{\s_{j,n}\vn^k(x_{j,n})}{\mn^k}\quad\left(\hbox{from \eqref{2.6b} and \eqref{4.23}}\right)
\nonumber\\
&=-2\log \ln \int_{B_R(x_{j,n})} \ln e^{\un}\vn^k\, dx
-(8\pi)^2\sum_{i=1}^mh_{ji}c_i^k+48 \pi c_j^k\log 2+o(1)
\label{4.27}
\end{align}
\end{remark}
\begin{proposition}\label{p4.4}
For any $k,j\in\{1,\dots,m\}$
we have
\[
\ln \int_{B_R(x_{j,n})} e^{\un}\frac{\vn^k(x)-\vn^k(x_{j,n})}{\mn^k}\, dx =-16 \pi c_j^k +o(1).
\]
\end{proposition}
\begin{proof}
Subtracting \eqref{4.26} by \eqref{4.27} we get the claim.
\end{proof}
\begin{proof}[Proof of Theorem \ref{t2}]
Set
\[
\tz_n:=\frac{\tv_{j,n}^k -v_n^k(x_{j,n})}{\mn^k}\quad \text{ in } B_{\frac R{\d_{j,n}}}(0).
\]
Then
\begin{align*}
-\Delta \tz _n=&-\frac {\Delta \tv_{j,n}^k}{\mn^k}=-\frac 1{\mn^k}\d_{j,n}^2 \Delta \vn(\d_{j,n}\tx+x_{j,n}) =e^{\tu_{j,n}}\tv_{j,n}^k
\end{align*}
so that
\begin{equation}\label{4.28}
-\Delta \tz _n=\mn^ke^{\tu_{j,n}}\tz _n+\vn^k(x_{j,n})e^{\tu_{j,n}}.
\end{equation}
The claim follows from elliptic estimates once we prove that
\begin{equation}\label{4.28a}
\tz_n=c_j^kU(\tx)+o(1)\quad \hbox{ locally uniformly in $\R^2$.}
\end{equation}

Using again the Green representation formula for \eqref{autov-n}, we have for $x\in \omega \subset\subset B_R(x_{j,n})$
\begin{align*}
\frac{\vn^k(x)}{\mn^k}=& \ln \intO G(x,y) e^{\un(y)} \vn^k(y) \, dy\\
=&\int_{B_{\frac R{\d_{j,n}}}(0)}\frac 1{2\pi} \log \frac 1{ | x-(\d_{j,n}\ty +x_{j,n})|} e^{\tu_n} \tv _{j,n}^k \, d\ty\\
&+8\pi c_j^kK(x,\kappa_j)+8\pi \sum_{\substack{1\leq i\leq m\\i\neq j}} c_i^kG(x,\kappa_i)+o(1)
\end{align*}
Therefore, letting $x=\d_{j,n}\tx+x_{j,n}$, we have for every $\tx\in \tilde\omega \subset \subset \R^2$ that
\begin{align*}
\frac{\tv _{j,n}^k(\tx)}{\mn^k} =&\frac 1{2\pi} \int_{B_{\frac R{\d_{j,n}}}(0)}\log \frac 1 {| \d_{j,n}\tx+x_{j,n}-\d_{j,n}\ty -x_{j,n}|} e^{\tu_{j,n}} \tv _{j,n}^k \, d\ty\\
&+8\pi c_j^k K(\d_{j,n}\tx+x_{j,n},\kappa_j) +8\pi \sum_{\substack{1\leq i\leq m\\i\neq j}} c_i^kG(\d_{j,n}\tx+x_{j,n} ,\kappa_i)+o(1)\\
=&\frac 1{2\pi} \log \frac 1{\d_{j,n}} \int_{B_{\frac R{\d_{j,n}}}(0)} e^{\tu_{j,n}} \tv _{j,n}^k \, d\ty+\frac 1{2\pi} \int_{B_{\frac R{\d_{j,n}}}(0)} \log \frac 1{|\tx-\ty|}e^{\tu_{j,n}} \tv _{j,n}^k \, d\ty\\
&+8\pi c_j^k R(\kappa_j)+8\pi \sum_{\substack{1\leq i\leq m\\i\neq j}} c_i^kG(\kappa_j,\kappa_i)+o(1)\quad\left(\hbox{using \eqref{4.23}}\right)\\
=&\frac{\vn^k(x_{j,n})} {\mn^k} +\frac 1{2 \pi}  \int_{B_{\frac R{\d_{j,n}}}(0)}\log \frac 1{|\tx-\ty|}e^{\tu_{j,n}} \tv _{j,n}^k \, d\ty+6c_j^k\log 2 +o(1)
\end{align*}
Then recalling the definition of $\tz_n$ we have
\[
\tz_n= \frac 1{2 \pi}  \int_{B_{\frac R{\d_{j,n}}}(0)}\log \frac 1{|\tx-\ty|}e^{\tu_{j,n}} \tv _{j,n}^k \, d\ty
+6c_j^k\log 2  +o(1)
\]
so that
\[
\tz_n= \frac 1{2 \pi} c _{j}^k \int_{\R^2}\log \frac 1{|\tx-\ty|}e^{U}  \, d\ty
+6c_j^k\log 2  +o(1)
\]
locally uniformly with respect to $\tx$ since $e^{\tu_{j,n}}=O(|\tx|^{-4})$ uniformly as $|\tx|\to \infty$.

Observe that
\[
\tilde \Psi(\tx):= \frac 1{2 \pi}  \int_{\R^2}\log \frac 1{|\tx-\ty|}e^{U} \, d\ty
\]
satisfy
\[
-\Delta \tilde \Psi=e^U\quad \hbox{ in }\mathfrak{D}'(\R^2).
\]
and it is a radially symmetric function. Then, since $-\Delta U=e^U$ and $U(0)=0$, we have $\tilde\Psi-\tilde\Psi(0)=U$ where $\tilde\Psi(0)=-6\log 2$, see \eqref{A_0}. Therefore $\tilde\Psi=U-6\log 2$. This implies that
$\tz_n \to c_j^kU$ locally uniformly and this proves \eqref{4.28a}. Finally, by Proposition \ref{p5.3} we have that the proof of Theorem  \ref{t2} is complete.
\end{proof}
\section{Proof of Theorems \ref{j2} and  \ref{j3}}
\label{s6}
\begin{proof}[Proof of Theorem \ref{j2}]
The final part of the proof of Proposition  \ref{p5.2} shows that, for any vector $\c^{ k}$, we have that at least $two$ components of $\c^{ k}$ are different from zero. This shows $ii)$.
Now we are going to prove $i)$.

We can assume that $v_n^1>0$ and then $c_j^1\ge0$ for any $j=1,..,m$. We want to prove that $c_j^1>0$ for any $j=1,..,m$ and so, by contradiction, let us assume that $c_1^1=0$ (the generic case is analogous). By  \eqref{5.1} we deduce that $c_h^1\sum_{i=2}^m h_{1i}c_i^1=0$. Since $\c^1\ne{\bf 0}$ there exists $h\ge2$ such that $c_h^1\ne0$. Moreover $h_{1i}<0$ for any $i\geq 2$ and this gives a contradiction.
\end{proof}
\begin{proof}[Proof of Theorem \ref{j3}]
Let us fix an integer $m>2$  and $\Omega=\left\{x\in\R^2\right.$
such that $\left. (0<)a<|x|<1\right\}$ . In \cite{NS90b}
there was constructed a $m$-mode solution $u_n$ to \eqref{1}, i.e.
a solution which is invariant with respect to a rotation of
$\frac{2\pi}m$ in $\R^2$,
\[
u(r,\theta)=u\left(r,\theta +
\frac{2\pi}m\right).
\]
Since it is not clear if the solution constructed in \cite{NS90b} blows-up at $m$ points, we refers to
\cite{EGP05} for an existence result of a $m$-mode solution verifying \eqref{2.1a}.

The $m$ blow-up points $\kappa_1=\cdots=\kappa_m$  are located on a circle concentric
with the annulus and are vertices of a regular polygon with $m$
sides. So we can assume that $\kappa_1=(r_0,0)$,
$\kappa_2=r_0\left(\cos\frac{2\pi}m, \sin\frac{2\pi}m\right)
,\dots, \kappa_m=r_0\left(\cos\frac{2(m-1)\pi}m, \sin\frac{2(m-1)\pi}m\right)$ for some $r_0\in(a,1)$.\\
Observe that since $G(x,\kappa_1)$ is symmetric with respect to
the $x_1$-axis, (see Lemma 2.1 in \cite{G02}), we get $G(\kappa_j,\kappa_1)=
G(\kappa_{m-j+2},\kappa_1)$, $j=2,..,m$.
Similarly the value $G(\kappa_i,\kappa_j)$ depends only on the distance between $\kappa_i$ and $\kappa_j$.
For example, $G(x,\kappa_2)=G(R_{-\frac{2\pi}{m}}x,\kappa_1)$ and consequently $G(\kappa_{i+1},\kappa_2)=G(\kappa_{i},\kappa_1)$, where $R_\theta$ denotes the rotation operator around $0$ with angle $\theta$. Similarly $G(\kappa_{i+k},\kappa_{1+k})=G(\kappa_{i},\kappa_1)$.
Note also that, if $\Omega$ is an annulus, the Robin function $R(x)$ is radial, so that $R(\kappa_1)=..=R(\kappa_m)=R$.

Here we set $G(\kappa_i,\kappa_1)=G_i$ and $R_l=R+4\sum_{h=2}^l G_h$. Then the matrix $h_{ij}$ becomes:
\\
if $m=2l$\;($l=1,2,\cdots$),
\[ (h_{ij})=
\begin{pmatrix}
R_l+2G_{l+1} & -G_2 & -G_3 & \dots& -G_{l+1}& \dots &-G_2\\
-G_2 & R_l+2G_{l+1} &-G_2&\dots&\dots&\dots& -G_3\\
& \dots\\
-G_2 & -G_3 &\dots&\dots&\dots&\dots& R_l+2G_{l+1}
\end{pmatrix},
\]
and for $m=2l+1$\;($l=1,2,\cdots$),
\[ (h_{ij})=
\begin{pmatrix}
R_l & -G_2 & -G_3 & \dots& -G_{l}& -G_{l}& \dots &-G_2\\
-G_2 & R_l &-G_2&\dots&\dots&\dots&\dots& -G_3\\
& \dots\\
-G_2 & -G_3 &\dots&\dots&\dots&\dots&-G_2& R_l
\end{pmatrix},
\]
A straightforward computation shows  that the first eigenvalue of $(h_{ij})$ is $\Lambda^1=R+2\sum_{h=2}^l G_h+G_{l+1}$ for $m=2l$ and $R+2\sum_{h=2}^l G_h$ for $m=2l+1$ which is simple. It is easy to see that the eigenspace corresponding to $\Lambda^1$ is spanned by $\c^1=(1,1,..,1)$.

Now consider separately the cases where $m$ is odd and $m$ is even.\\
{\em Case 1: $m$ is odd}.

Let $v_n^k$ be an eigenfunction related to the eigenvalue
$\mu_n^k$ with $k\ge2$ and rotate it by an angle of $\frac{2\pi}m$. By the
symmetry of the problem we get that the rotated function $\bar
v_n^k(r, \theta)=v_n^k\left(r, \theta +\frac{2\pi}m\right)$ is still an
eigenfuction related to the same eigenvalue $\mu_n^k$. If by contradiction
the eigenvalue $\mu_n^k$ is simple we have that 
\begin{equation}\label{z1}
\bar v_n^k=\alpha v_n^k,
\end{equation}
for some $\alpha\ne0$.

Let $\bar\c^k$ the eigenvector given by \eqref{matrix_h} associated to $\bar\mu_n^k$. Denoting by $\c^k=(c_1^k,\dots,c_m^k)$ the
 eigenvector associated to $\mu_n^k$ we have, by the definition of $\bar v_n^k$,
\begin{equation}\label{z2}
\bar\c^k=\left(c_2^k,c_3^k\dots,c_m^k,c_1^k\right).
\end{equation}
By  \eqref{z1} and  \eqref{z2} we derive that
\begin{equation}\label{z3}
\alpha c_i^k=c_{i+1}^k\quad\hbox{for }i=1,\dots,m,\hbox{ meaning that }c_{m+1}=c_1.
\end{equation}
From \eqref{z3} we get that $c_i^k=\alpha^mc_i^k$. Since $\c^k\ne0$ we get $\alpha^m=1$
and since $m$ is odd we derive  that $\c^k=(1,1,..,1)=\c^1$.
This gives a contradiction since $k\ge2$.\\
{\em Case 2: $m$ is even}.\\
Let $v_n^k$ be an eigenfunction related to the eigenvalue
$\mu_n^k$ with $k\ge2$ and define $\bar v_n^k$ as in the previous case. Repeating step by step the proof, assuming that  $\mu_n^k$ is a simple eigenvalue, we again deduce that $\alpha^m=1$. However, since in this case $m$ is even, we have the solution $\alpha=-1$  and by \eqref{z3} we get
$\c^k=(-1,1,-1,1,..,-1,1)$ and the corresponding eigenvalue $\Lambda^k$ is given by $\Lambda^k=R+\left(2+(-1)^{\frac{m+2}2}\right)G_{\frac{m+2}2}+2\sum\limits_{h=2}^l\left(2+(-1)^h\right)G_h$. Hence $\mu_n^k$ is the {\em unique} simple eigenvalue. This gives the claim.
\end{proof}


\begin{thebibliography}{99}
\bibitem[B80]{B80} Bandle, C.: Isoperimetric Inequalities and Applications. Pitman Publishing, London (1980)


\bibitem[CLMP92]{clmp92}
Caglioti, E., Lions, P.-L., Marchioro C., Pulvirenti M.:
A special class of stationary flows for two-dimensional Euler equations: a statistical mechanics description.
Comm. Math. Phys. {\bf 143}, 501--525 (1992)

\bibitem[CLMP95]{clmp95}
Caglioti, E., Lions, P.-L., Marchioro C., Pulvirenti M.:
A special class of stationary flows for two-dimensional Euler equation: a statistical mechanics description, Part II.
Comm. Math. Phys. {\bf 174}, 229--260 (1995)

\bibitem[CL91]{CL91} Chen, W., Li, C.: Classification of solutions of some nonlinear elliptic equations. Duke Math. J. {\bf 63}, 615--622 (1991)

\bibitem[CL02] {CL02} Chen, C.C., Lin, C.S.: Sharp estimates for solutions of multi-bubbles in compact Riemann surfaces. Comm. Pure Appl. Math. {\bf 55}, 728--771 (2002)

\bibitem[DKM05]{DKM05}
 del Pino M., Kowalczyk M, Musso M., Singular limits in Liouville-type equations, Calc. Var. PDE, {\bf 24}, (2005) 47--81.

\bibitem[EGP05] {EGP05} Esposito, P., Grossi, M., Pistoia A.: On the existence of blowing-up solutions for a mean field equation. {\bf 22}, 227-–257 (2005)

\bibitem[GG09] {GG09}
Gladiali, F., Grossi, M.:
On the spectrum of a nonlinear planar problem.
Ann. Inst. H. Poincar\'e Anal. Non Lin\'eaire {\bf 26}, 728--771 (2009)

\bibitem[G02]{G02} Grossi,M.: On the nondegeneracy of the critical points
of the Robin function in symmetric domains. C. R. Acad. Sci.
Paris, Ser. I {\bf 335} 157--160 (2002)

\bibitem[GOS11]{GOS11}
Grossi, M., Ohtsuka, H., Suzuki, T.:
Asymptotic non-degeneracy of the multiple blow-up solutions to the Gel'fand problem in two space dimensions.
Adv. Differential Equations {\bf 16}, 145--164 (2011)

\bibitem[GGOS12]{GGOS12} Gladiali, F., Grossi, M., Ohtsuka, H., Suzuki, T.:
Morse indices of multiple blow-up solutions to the two-dimensional Gel'fand problem. {\tt arXiv:1210.1373}, 39pages, (2012).

\bibitem[JT80]{JT80} Jaffe, A., Taubes, C.: Vortices and Monopoles, Structure of Static Gauge Theories. Birkh{\"a}user, Boston (1980)

\bibitem[JM73]{jm73}
Joyce, G., Montgomery, D.:
Negative temperature states for two-dimensional guiding-center plasma.
J. Plasma Phys. {\bf 10}, 107--121 (1973)

\bibitem[K93]{kie93}
Kiessling, M.K.H.:
Statistical mechanics of classical particles with logarithmic interaction.
 Comm. Pure Appl. Math. {\bf 46}, 27--56 (1993)


\bibitem[Li99]{YYL99}
Li, Y. Y.:
Harnack type inequality: the method of moving planes.
Comm. Math. Phys. {\bf 200}, 421--444 (1999)


\bibitem[NS90]{NS90}
Nagasaki, K., Suzuki, T.:
Asymptotic analysis for two-dimensional elliptic eigenvalues problems with exponentially dominated nonlinearities.
Asymptotic Analysis {\bf 3},  173--188 (1990)

\bibitem[NS90b]{NS90b}
Nagasaki, K., Suzuki, T.: Radial and nonradial solutions for the
nonlinear eigenvalue problem $\Delta u+\l e^u=0$ on annuli in
$\R^2$. J. Differential Equations {\bf 87} 144--168 (1990)


\bibitem[O12]{O12}
Ohtsuka, H.:
To what extent can the Hamiltonian of vortices illustrate the mean field of equilibrium vortices?
RIMS {K{\^o}ky{\^u}roku} {\bf 1798}, 1--17 (2012)

\bibitem[PL76]{pl76}
Pointin, Y.B., Lundgren, T.S.:
Statistical mechanics of two-dimensional vortices in a bounded container.
Phys. Fluids {\bf 19}, 1459--1470 (1976)

\bibitem[SS00]{SS00} Senba, T., Suzuki, T.: Some structures of the solution set for a stationary system of chemotaxis. Adv. Math. Sci. Appl. {\bf 10}, 191--224 (2000)

\bibitem[S08]{suzuki08}
Suzuki, T.:
Mean Field Theories and Dual Variation.
Atlantis Press, Amsterdam-Paris (2008)

\end{thebibliography}
\end{document}